\documentclass[12pt,reqno]{amsproc}

\usepackage{amsmath,amsthm,amssymb,wasysym}
\usepackage[inline]{enumitem}
\usepackage{caption}
\usepackage{subcaption}
\usepackage{graphicx}

\DeclareMathOperator{\RE}{Re}

\numberwithin{equation}{section}
\theoremstyle{plain}

\newtheorem{theorem}{Theorem}[section]

\newtheorem{remark}[theorem]{Remark}

\theoremstyle{definition}
\newtheorem{definition}[theorem]{Definition}

\allowdisplaybreaks

\begin{document}
\title{Radius of Starlikeness of Certain Analytic Functions}

\author[A. Sebastian]{Asha Sebastian}
\address{Department of Mathematics \\National Institute of Technology\\Tiruchirappalli-620015,  India }
\email{ashanitt18@gmail.com}

\author{V. Ravichandran}
\address{Department of Mathematics \\National Institute of Technology\\Tiruchirappalli-620015,  India }
\email{vravi68@gmail.com; ravic@nitt.edu}

\begin{abstract}
This paper  studies analytic functions $f$ defined on the open unit disk of the complex plane  for which $f/g$ and $(1+z)g/z$ are both functions with positive real part for some  analytic function $g$. We determine radius constants of these functions to  belong  to classes of strong starlike functions,  starlike functions of order $\alpha$,   parabolic starlike functions,  as well as to the classes of  starlike functions associated with lemniscate of Bernoulli, cardioid,  lune,  reverse lemniscate,  sine function,  exponential function and  a particular rational function. The  results obtained are sharp.
\end{abstract}


\thanks{The first author is supported by  the institute  fellowship from NIT Tiruchirappalli.}

\maketitle
\section{Introduction}
The unit disk $\mathbb{D}$ consists of all points $z\in \mathbb{C}$ satisfying $|z|<1$ and the class  $\mathcal{A}$ consists of all analytic functions $f:\mathbb{D}\to \mathbb{C}$ normalised by  the condition $f(0) = f'(0)-1 = 0$.  Let $\mathcal{S}$ denote the  class of univalent functions  in $\mathcal{A}$. For two families $\mathcal{G}$ and $\mathcal{F} $ of $\mathcal{A}$,  the $\mathcal{G}$-radius of $\mathcal{F} $  denoted by $R_{\mathcal{G}}(\mathcal{F} )$ is the largest number $R$ such that $r^{-1}f(rz) \in \mathcal{G}$ for $ 0< r \leq R$  and for all $f \in \mathcal{F}$.   Whenever $\mathcal{G}$ is characterised by a geometric property $P$,  the number $R$ is also referred to as the radius of property $P$ for the class $\mathcal{F} $. There are various studies on radius problems; one of these studies focus on functions $f$  characterised by the ratio between $f$ and  another function $g$ belonging to particular subclasses of $\mathcal{A}$ \cite{greg, greg1, greg2, ratti, ratti1, shah, shan}. Ali \emph{et al.}\  obtained several radius  results for classes of functions $f$ satisfying one of the conditions: (i) $\operatorname{Re}(f(z)/g(z))> 0 $ where $\operatorname{Re} (g(z)/z) > 0$ or $\operatorname{Re} (g(z)/z) > 1/2$ (ii) $\left|
(f(z)/g(z))-1\right|< 1$ where $\operatorname{Re} (g(z)/z) > 0$ or $g$ is convex. These classes are closely related to the class of close-to-convex functions and for functions belonging to several related classes we study various radii for these functions to belong to several subclasses of starlike functions discussed below.

An analytic function $f$ is \emph{subordinate} to another analytic function  $g$,    written   $f \prec g$ or $f(z) \prec g(z)$ $ (z \in \mathbb{D})$,  if there exists an analytic function $w:\mathbb{D}\to \mathbb{D}$ with $w(0)=0$ satisfying $f=g\circ w$. For a  univalent function $g$,   the equivalent condition for subordination is $f(0)=g(0)$ and $f(\mathbb{D}) \subseteq g(\mathbb{D})$.   Subclasses of starlike and convex functions are characterised by the quantities $z f'(z)/f(z)$ or $1+z f''(z)/f'(z)$ lying in a convex region in the right half plane. In 1989,    Shanmugam \cite{shan1} unified the classes of starlike and convex functions using subordination and convolution: for a fixed $g\in \mathcal{A}$ and a convex function $\varphi$ with $\varphi(0)=1$, he considered the class $\mathcal{S}^*_g(\varphi)$ of all functions $f\in \mathcal{A}$ satisfying $z(f*g)'(z)/(f*g)(z)\prec \varphi(z)$. Here the function $ f*g$ denotes the \emph{convolution} (or \emph{the Hadamard product}) of two analytic functions $f(z)=\sum_{n=1}^\infty a_n z^n$ and $g(z)=\sum_{n=1}^\infty b_nz^n$ defined by
$(f*g)(z)=\sum_{n=1}^\infty a_nb_nz^n$. When $g(z)=z/(1-z)$,  we denote the class $\mathcal{S}^*_g(\varphi)$  by $\mathcal{S}^* (\varphi)$ and when $g(z)=z/(1-z)^2$,   the class $\mathcal{S}^*_g(\varphi)$  by $\mathcal{K}(\varphi)$. In 1992,  Ma and Minda \cite{ma} gave a unified treatment of growth,  distortion and covering theorems of functions in the classes $\mathcal{S}^*(\varphi)$ and $\mathcal{K}(\varphi)$. Note that $\mathcal{S}^{*}[A,  B]:= \mathcal{S}^{*}((1+Az)/(1+Bz)), \quad -1 \leq B < A \leq 1$  is the class of Janowski starlike functions and $\mathcal{K}[A,  B]:= \mathcal{K}((1+Az)/(1+Bz))$ is the class of Janowski convex functions. For  $0 \leq \alpha < 1$, the classes $  \mathcal{S}^{*}(\alpha)=  \mathcal{S}^{*}[1-2\alpha, -1]$ and $\mathcal{K}(\alpha)= \mathcal{K}[1-2\alpha, -1]$ are the familiar classes of  starlike functions of order $\alpha$ and convex functions of order $\alpha$ respectively. These classes were studied extensively in \cite{dur, good, jan1, jan2}. The class $\mathcal{UCV}$ of \emph{uniformly convex functions} consists of all functions $f \in \mathcal{A}$ that maps every circular arc contained in $\mathbb{D}$ with centre $\xi \in \mathbb{D}$ onto a convex arc. R\o nning \cite{ronn} and Ma and Minda \cite{ma1} independently proved that a function $f \in \mathcal{UCV}$ if
	\[1+\frac{z f''(z)}{f'(z)} \prec 1+ \frac{2}{\pi^{2}}\left(\log \frac{1+\sqrt{z}}{1-\sqrt{z}}\right)^{2} :=\varphi_{PAR}(z).\]
The image $\varphi_{PAR}(\mathbb{D})=\{w = u+\operatorname{i}v: v^{2} < 2 u-1\} = \{w: \operatorname{Re}w > \left|w-1 \right| \}$ is a parabolic region and the functions in  the class $\mathcal{S}_{p}:=\mathcal{S}^*(\varphi_{PAR})$  are known as \emph{parabolic starlike functions}. These functions are studied by several authors \cite{ali2,ganga, ma1, ravi, shan}. For $0 < \gamma \leq 1$,  the class $\mathcal{S}^*((1+z/1-z)^{\gamma})$  is the class $\mathcal{S}_{\gamma}^{*}$ of \emph{strongly starlike functions of order} $\gamma$ and note that the function $f\in\mathcal{S}_{\gamma}^{*}$ if    $\left|\operatorname{arg} (zf'(z)/f(z))\right| \leq \pi \gamma/2$.

In 1996,  Sok\'{o}\l{} and Stankiewicz \cite{sokol} and several authors \cite{ali, ali1, aouf, pap} studied the class $\mathcal{S}_{L}^{*} = \mathcal{S}^{*}(\sqrt{1+z})$.  Geometrically,  the class $\mathcal{S}_{L}^{*}$ represents a collection of functions $f \in \mathcal{A}$ such that $zf'(z)/f(z)$ lies in the region bounded by the right half of the lemniscate of Bernoulli $\left|w^{2}-1\right|= 1$ or  $ (u^{2}+v^{2})^{2}-2(u^{2}-v^{2})=0$.   Motivated by these,  Mendiratta \emph{et al.}\ \cite{rajni1, rajni} introduced and studied the sub classes of starlike functions
\[\mathcal{S}_{e}^{*}= \mathcal{S}^{*}(e^{z}) \quad\text{and} \quad \mathcal{S}_{RL}^{*}=\mathcal{S}^{*}\left(\sqrt{2}-(\sqrt{2}-1)\sqrt{\frac{1-z}{1+2(\sqrt{2}-1)z}}  \right).\] Geometrically, the function $f \in \mathcal{S}_{RL}^{*} $ provided $zf'(z)/f(z)$ lies in the interior of the left half of the shifted lemniscate of Bernoulli given by $\left|(w-\sqrt{2})^{2}-1\right| < 1$.   Similarly, Sharma \emph{et al.}\ \cite{kanika} studied various properties of the class $\mathcal{S}_{c}^{*}= \mathcal{S}^{*}(\varphi_{c}(z))$ where $\varphi_{c}(z)=1+(4/3)z+(2/3)z^{2}$.   Essentially, a function $f \in \mathcal{S}_{c}^{*} $ if $zf'(z)/f(z)$ lies in the region bounded by the cardioid $\Omega_{c}:=\{x+\operatorname{i}y:(9x^{2}+9y^{2}-18x+5)^{2}-16(9x^{2}+9y^{2}-6x+1)=0\}$.   In 2015,  Raina and Sok\'{o}\l{} \cite{raina} considered the class $\mathcal{S}_{\leftmoon}^{*} = \mathcal{S}^{*}(h)$ where $h(z) = z + \sqrt{1 + z^{2}}$.     They proved that $f \in \mathcal{S}_{\leftmoon}^{*}$ iff $zf'(z)/f(z)$ belongs to the lune shaped region $\mathcal{R}:=\{w \in \mathbb{C}: \left|w^{2}-1\right| < 2 |w|\}$.   Several other properties of $\mathcal{S}_{\leftmoon}^{*}$ were discussed by Gandhi and Ravichandran \cite{gandhi}. Kumar and Ravichandran \cite{sushil} considered the class $\mathcal{S}_{R}^{*}= \mathcal{S}^{*}(\psi)$ where $\psi(z)=1+(zk+z^2/(k^2-kz)),  \quad k=\sqrt{2}+1$.   Cho \emph{et al.}\  \cite{cho} in a similar fashion defined the class $\mathcal{S}_{sin}^{*}= \mathcal{S}^{*}(1+ \sin z)$. Recently,  several interesting subclasses of  starlike functions were studied when $\varphi$ is related to Booth lemniscate \cite{booth} and  Bell numbers \cite{hm}.

Kotsur \cite{kotsur1} investigated the class of functions $f \in \mathcal{A}$ such that $\operatorname{Re} (zf'(z))'/g'(z)>0$ where $g$ is starlike. Similar studies on radii of starlikeness and convexity of certain close to convex functions can be found in \cite{kotsur2, kotsur3, kotsur4, kotsur5}. Kowalczyk and Lecko \cite{kowa1, kowa2} introduced polynomial close to convex functions and determined radii of selected classes in them. For further related results, see \cite{kowa3, kowa4, lecko1}. Pranav Kumar and Vasudevarao \cite{vasudevarao} estimated the logarithmic coefficients of certain close to convex functions where $\operatorname{Re} (1-z)f'(z)>0$,  $\operatorname{Re}(1-z^2)f'(z)>0$ and $\operatorname{Re} (1-z+z^2)f'(z)>0$. Recently in 2019, Lecko and Sim \cite{lecko2} considered the starlike functions $z/(1-z^2)$, $z/(1-z)^2$ and investigated the functions satisfying $\operatorname{Re} (1-z^2)f(z)/z>0$ and  $\operatorname{Re} (1-z)^2 f(z)/z>0$. They have also determined certain sharp coefficient estimates.

Motivated by these studies,  we define certain classes of functions $f \in \mathcal{A}$ in the open unit disc $\mathbb{D}$ characterised by its ratio with a certain function $g$. The classes we discuss here consist of functions $f \in \mathcal{A}$ satisfying the following conditions: (i) $\operatorname{Re}f(z)/g(z)>0$ where $\operatorname{Re}(1+z)g(z)/z>0$ (ii) $\left|f(z)/g(z) - 1\right|>0$ where $\operatorname{Re}(1+z)g(z)/z>0$ (iii) $\operatorname{Re}(1+z)f(z)/z>0$ (iv) $\operatorname{Re}(1+z)^2f(z)/z>0$. We compute radius constants of the above functions for several interesting subclasses of $\mathcal{A}$ like starlike functions of order $\alpha$, parabolic starlike functions, starlike functions associated with lemniscate of Bernoulli, exponential function, cardioid, sine function, lune, a particular rational function, reverse lemniscate and strong starlike functions. The main technique involved here in finding the radius for the classes of functions is to determine the disk that contains the image of $\mathbb{D}$ by the mapping $zf'(z)/f(z)$.

\section{Radius Problems}
Let $\mathcal{P}$ be the class of all analytic functions $p:\mathbb{D}\to\mathbb{C}$ with $p(0)=1$ and $\RE p(z)>0$ for all $z\in \mathbb{D}$. This class is known as the class of functions with positive real part or the class of  Carath\'eodory functions.  This class is used in characterization of several well-studied classes of univalent functions. Our first result concerns the class  $\mathcal{F}_{1}$   of functions $f \in \mathcal{A}$ satisfying   $\ f/g\in \mathcal{P} $ for some $g \in \mathcal{A}$ with $ (1+z)g(z)/z \in \mathcal{P} $. The functions $ f_1,  g_1: \mathbb{D} \longrightarrow \mathbb{C}$ defined by
\begin{equation}\label{extremal}
f_1(z)=\frac{z(1-z)^{2}}{(1+z)^{3}}  \quad\text{and}\quad  g_1(z)=\frac{z(1-z)}{(1+z)^{2}}
\end{equation}
satisfy
\[\operatorname{Re}\frac{f_1(z)}{g_1(z)}=\operatorname{Re}\frac{(1+z)g_1(z)}{z}
=\operatorname{Re}\frac{1-z}{1+z}>0,\] and hence the function $f_1\in \mathcal{F}_{1}$. This proves that the class $ \mathcal{F}_{1}$ is non-empty. This function $f_1$ is the extremal function for the radius problems that we consider. As
\[ f_1(z)=z-5 z^2+13 z^3-25 z^4+\dotsc,\]
the functions in  $ \mathcal{F}_{1}$ are not necessarily  univalent.
Since
\[ f_1'(z)=\frac{(1-z)(1-5z)}{(1+z)^4},\]  we have $f_1'(1/5)=0$ and
it follows, by the first part of the following theorem,  that the radius of univalence of the functions in class $\mathcal{F}_{1}$ is 1/5.

\begin{theorem}
For the class $\mathcal{F}_{1}$,  the following results hold:
\begin{enumerate}[label=(\roman*)]
\item The $\mathcal{S}^{*}(\alpha)$ radius
$R_{\mathcal{S}^{*}(\alpha)}=2(1-\alpha)/\left(5+\sqrt{25-(4\alpha(1-\alpha))}\right)$$, \quad 0 \leq \alpha< 1$.	
	
\item The $\mathcal{S}_{L}^{*}$ radius
$R_{\mathcal{S}_{L}^{*}}= (2\sqrt{2}-2)/ \left(5+\sqrt{33-4\sqrt{2}}\right)  \approx 0.0809$.

\item The $\mathcal{S}_{p}$ radius
$R_{\mathcal{S}_{p}}= 5-2\sqrt{6} \approx 0.1010$.

\item The $\mathcal{S}_{e}^{*}$ radius
$R_{\mathcal{S}_{e}^{*}}= (2e-2)/\left(5e+\sqrt{25e^{2}+4(1-e) }\right)\approx 0.1276$.

\item The $\mathcal{S}_{c}^{*}$ radius
$R_{\mathcal{S}_{c}^{*}}= (15-\sqrt{217})/2 \approx 0.1345$.

\item The $\mathcal{S}_{sin}^{*}$ radius
$R_{\mathcal{S}_{sin}^{*}}= 2\sin 1 /\left(5+\sqrt{25+4 \sin 1 (1+ \sin 1)}\right)\approx 0.1589$.

\item The $\mathcal{S}_{\leftmoon}^{*}$ radius
$R_{\mathcal{S}_{\leftmoon}^{*}}= (4-2\sqrt{2})/\left(5+\sqrt{41-12\sqrt{2}}\right) \approx 0.1183$.

\item The $\mathcal{S}_{R}^{*}$ radius
$R_{\mathcal{S}_{R}^{*}}= (6-4\sqrt{2})/\left(5+\sqrt{81-40\sqrt{2}}\right) \approx 0.0342$.

\item The $\mathcal{S}_{RL}^{*}$ radius
$R_{\mathcal{S}_{RL}^{*}}$ is the root $(\approx 0.0566)$ in $[0,1]$ of the equation
\begin{align*}
25r^{2}+ (r^{2}-1)^2 +(r^2-1)\sqrt{(r^2+\sqrt{2})(2-\sqrt{2}-r^2)} & \\ - (1+ \sqrt{2}(r^2-1))^2  & = 0.
\end{align*}

\item The $\mathcal{S}_\gamma^{*}$ radius of strong starlikeness $R_{\mathcal{S}_\gamma^{*}} \geq \sin(\pi\gamma/2)/5,  \quad 0 < \gamma \leq 1$.

\end{enumerate}
\end{theorem}
\begin{proof}

 Let the function $ f \in \mathcal{F}_{1}$. Let the function $ g: \mathbb{D} \longrightarrow \mathbb{C}$ be chosen such that
\begin {equation} \label{g}
\operatorname{Re}\frac{f(z)}{g(z)}>0 \quad \text{and}\quad
  \operatorname{Re}\left (\frac{1+z}{z}g(z)\right )>0  \quad (z \in \mathbb{D}).
\end{equation}
Define the functions $p_{1}, p_{2}:\mathbb{D} \longrightarrow \mathbb{C} $ by
\begin{equation}\label{h}
p_{1}(z)=\frac{1+z}{z}g(z)\quad \text{and} \quad p_{2}(z)= \frac{f(z)}{g(z)}.
\end{equation}
By (\ref{g}) and (\ref{h}), we have $p_{1}, p_{2} \in \mathcal{P}$
and $f(z)= z p_{1}(z) p_{2}(z)/(1+z)$.
Then,  a calculation involving logarithmic derivative of the function $f$ shows that
\begin{equation} \label{main1}
\frac{zf'(z)}{f(z)}=\frac{z p_{1}^{'}(z)}{p_{1}(z)}+\frac{z p_{2}^{'}(z)}{p_{2}(z)}+\frac{1}{1+z}.
\end{equation}
The bilinear transformation $1/(1+z)$ maps the disk $\left|z\right| \leq r$ onto the disk
\begin{equation}\label{disk}
 \left | \frac{1}{1+z}-\frac{1}{1-r^{2}} \right | \leq \frac{r}{1-r^{2}}.
\end{equation}
For $p \in \mathcal{P}(\alpha):=\{p \in \mathcal{P}| \operatorname{Re} p > \alpha \}$, by  \cite[Lemma 2]{shah}, we have
\begin{equation}\label{shah}
	\left | \frac{z p'(z)}{p(z)} \right | \leq \frac{2(1-\alpha)r}{(1-r)(1+(1-2\alpha)r)} \quad  (\left| z \right| \leq r).
\end{equation}
Using (\ref{disk}) and (\ref{shah}), it follows from \eqref{main1} that the function $f$ maps the disk $\left|z\right| \leq r$ onto the disk
\begin{equation} \label{star1}
\left | \frac{zf'(z)}{f(z)} - \frac{1}{1-r^{2}}  \right |  \leq \frac{5r}{1-r^{2}}.
\end{equation}
The classes we discuss here are all subclasses of starlike functions. These classes are described by the quantity $zf'(z)/f(z)$ lying in some region in the right half plane. The radius problems are solved by finding $r$ such that the disk in \eqref{star1} is contained in the corresponding regions. By \eqref{star1}, we have
\begin{equation}\label{starlike}
	\operatorname{Re} \frac{zf'(z)}{f(z)} \geq \frac{1-5r}{1-r^{2}} \geq 0,\quad (r \leq 1/5),
\end{equation}
then the function $ f \in \mathcal{F}_{1}$ is starlike in $\left|z\right| \leq 1/5$.  Hence all the radii that we estimate will be less than $1/5$. Note that,  for $0 < r \leq 1/5$,  the centre of disk in (\ref{star1}) lies in the interval $[1,  25/24] \approx [1, 1.0416]$.

\begin{enumerate}[label=(\roman*),  leftmargin=12pt]
\item The number $r=R_{\mathcal{S}^{*}(\alpha)} $ is the root of
$  \alpha r^{2}-5r+1-\alpha=0$ in $[0,1]$ and hence, for $0<r\leq R_{\mathcal{S}^{*}(\alpha)}$, it follows from (\ref{starlike}) that
\[ 	\operatorname{Re} \frac{zf'(z)}{f(z)} \geq   \frac{1-5r}{1-r^{2}}\geq \alpha.\]
For the functions  $f_{1} \in \mathcal{F}_{1}$  given by \eqref{extremal} we have
\[\frac{zf_{1}^{'}(z)}{f_{1}(z)}=\frac{1-5z}{1-z^{2}}=\frac{1-5r}{1-r^{2}}=\alpha, \quad (z=r= R_{\mathcal{S}^{*}(\alpha)})\]
and this shows that the radius is sharp.
	
\item	  It   follows from (\ref{star1}) that
\begin{equation} \label{estimate}
\left | \frac{zf'(z)}{f(z)} - 1  \right |  \leq \left | \frac{zf'(z)}{f(z)} - \frac{1}{1-r^{2}}  \right | + \frac{r^2}{1-r^2} \leq \frac{5r+ r^{2}}{1-r^{2}}.
\end{equation}The number $r= R_{\mathcal{S}_{L}^{*}}$ is the root in $[0,1]$ of $(5r+r^2)=(\sqrt{2}-1)(1-r^2)$ and for $0 < r \leq R_{\mathcal{S}_{L}^{*}}$,  we have
\begin{equation} \label{lemniradius}
	\frac{5r+r^{2}}{1-r^{2}} \leq \sqrt{2}-1.
\end{equation}
Therefore,  by (\ref{estimate}) and (\ref{lemniradius}), for $0 < r \leq R_{\mathcal{S}_{L}^{*}}$, we have
\begin{equation}\label{connect}
	\left | \frac{zf'(z)}{f(z)} - 1  \right |  \leq \frac{5r+ r^{2}}{1-r^{2}} \leq \sqrt{2}-1.
\end{equation} For $0 < r \leq R_{\mathcal{S}_{L}^{*}}$,  using triangle inequality and \eqref{connect}, we have
\begin{equation}\label{tconnect}
	\left | \frac{zf'(z)}{f(z)} + 1  \right |\leq 2+  \left | \frac{zf'(z)}{f(z)} - 1  \right | \leq \sqrt{2}+1
\end{equation}
 and hence by \eqref{connect} and \eqref{tconnect},
\[ \left|\left(\frac{zf'(z)}{f(z)}\right)^{2} - 1\right|\leq \left|\frac{zf'(z)}{f(z)} + 1\right|\left|\frac{zf'(z)}{f(z)} - 1\right|\leq (\sqrt{2}+1)(\sqrt{2}-1)=1.\]
%
%
%
The number  $\rho=  R_{\mathcal{S}_{L}^{*}}$ satisfies $(1+5\rho)/(1-\rho^{2})= \sqrt{2}$. Using this, we see that the function $f_{1}$ defined in (\ref{extremal}) satisfies
\[ \left| \left(\frac{zf_{1}'(z)}{f_{1}(z)}\right)^{2} - 1    \right|=   \left| \left(\frac{1-5z}{1-z^{2}}\right )^{2}-1\right|= \left| \left(\frac{1+5\rho}{1-\rho^{2}}\right)^{2} - 1    \right|=1, \quad (z:= -\rho= -R_{\mathcal{S}_{L}^{*}}).\]
This shows that the radius is sharp (See Figure \ref{fig:pb1}.(\subref{fig1:pb1a})).

\item
Let $\Omega_{PAR}=\{w=u+\operatorname{i}v:v^{2}<2u-1\}=\{w: \operatorname{Re}w > \left|w-1\right| \}$.  Note that $\Omega_{PAR}$ is the interior of a parabola in the right half plane which is symmetric about real axis and has vertex at $\left(1/2,  0\right)$.  By Lemma \cite[pp.321]{shan}, for $1/2 < a < 3/2$, we have
\begin{equation}\label{parabolic}
\{w \in \mathbb{C}: \left | w-a \right |< a-1/2  \}\subseteq \Omega_{PAR}.
\end{equation}
If $0 < r \leq R_{\mathcal{S}_{p}}$,  then $a= 1/(1-r^{2}) \leq 3/2$ and
\[\frac{5r}{1-r^{2}} \leq \frac{1}{1-r^{2}}-\frac{1}{2}.\] Thus,  by \eqref{parabolic},  we see that the disk in (\ref{star1}) lies inside the parabolic region $\Omega_{PAR}$.  Sharpness follows for the function $f_{1}$ defined in (\ref{extremal}) (See Figure \ref{fig:pb1}.(\subref{fig1:pb1b})).
At $z:=\rho= R_{\mathcal{S}_{p}}$, we have
\[\operatorname{Re}\frac{zf'(z)}{f(z)}=\frac{1-5\rho}{1-\rho^2}=\frac{5\rho-\rho^2}{1-\rho^2}=\left|\frac{\rho^2-5\rho}{1-\rho^2}\right|=\left|\frac{zf'(z)}{f(z)}-1\right|.\]

\begin{figure}
	\centering
	\begin{subfigure}{.5\textwidth}
		\centering
		\includegraphics[height=3cm,keepaspectratio]{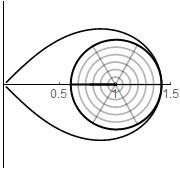}
		\caption{Sharpness of class $\mathcal{S}_{L}^{*}$}
		\label{fig1:pb1a}
	\end{subfigure}%
	\begin{subfigure}{.5\textwidth}
		\centering
		\includegraphics[height=3cm,keepaspectratio]{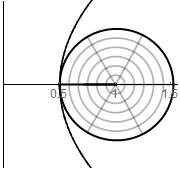}
		\caption{Sharpness of class $\mathcal{S}_p$}
		\label{fig1:pb1b}
	\end{subfigure}
	\caption{Sharpness of starlike functions associated with lemniscate and parabolic starlike functions.}
	\label{fig:pb1}
\end{figure}

\item For  $   e^{-1} \leq a \leq (\ e+\ e^{-1})/2$, by \cite[Lemma 2.2]{rajni}, we have
\begin{equation}\label{exponential}
	\{w \in \mathbb{C}: \left | w-a \right |< a-\ e^{-1}  \}\subseteq \{w \in \mathbb{C}: \left | \log w\right |< 1 \}=: \Omega_{e}.
\end{equation}  For $0 < r \leq R_{\mathcal{S}_{e}^{*}}$, we have
  $1/e \leq  a= 1/(1-r^{2}) \leq (\ e+\ e^{-1})/2$ and
	\[\frac{5r}{1-r^{2}} \leq \frac{1}{1-r^{2}}-\frac{1}{e}.\]
By  \eqref{exponential}, the disk in \eqref{star1}  lies inside
$ \Omega_{e} $   for $0 < r \leq R_{\mathcal{S}_{e}^{*}}$ proving   that  the $\mathcal{S}_{e}^{*}$ radius for the class $\mathcal{F}_{1}$ is $R_{\mathcal{S}_{e}^{*}}$.
The sharpness follows for the function $f_{1}$ defined in (\ref{extremal}) (See Figure \ref{fig:pb1a}.(\subref{fig1:sub3}).
Indeed at $z:=\rho= R_{\mathcal{S}_{e}^{*}}$, we have  \[ \left| \log \frac{zf_{1}'(z)}{f_{1}(z)} \right|= \left| \log \frac{1-5 \rho}{1-\rho^{2}}\right|=1.\]

\item 	For $1/3 < a \leq 5/3$, by \cite[ Lemma 2.5]{kanika}, we have
\begin{equation}\label{cardiod}
	\{w \in \mathbb{C}: \left | w-a \right |<(3a-1)/3 \}\subseteq \Omega_{c}
\end{equation}
where $\Omega_{c}$ is the region bounded by the cardioid $\{x+iy: (9x^{2}+9y^{2}-18x+5)^{2}-16(9x^{2}+9y^{2}-6x+1)=0\}$.
If $0 < r \leq R_{\mathcal{S}_{c}^{*}}$, then
\[\frac{5r}{1-r^{2}} \leq  \frac{1}{1-r^{2}}-\frac{1}{3} .\]  By (\ref{cardiod}), we see that the disk in (\ref{star1}) lies inside $\Omega_{c}$,  if $0 < r \leq R_{\mathcal{S}_{c}^{*}}$. The  result is sharp for the function $f_1$ defined in (\ref{extremal}). At $z:=\rho= R_{\mathcal{S}_{c}^{*}}$,
\[\left|\frac{zf'(z)}{f(z)}\right|=\left|\frac{1-5\rho}{1-\rho^2}\right|=\frac{1}{3}=\Omega_{c}(-1)\in\partial_c(\mathbb{D}).\]

\item For $\left|a-1\right| \leq \sin 1$, by \cite[Lemma 3.3]{cho}, we have
\begin{equation}\label{sine}
	\{w \in \mathbb{C}: \left | w-a \right |< \sin 1-\left|a-1\right|\}\subseteq \Omega_{s}
\end{equation}
where $\Omega_{s}:=q_{0}(\mathbb{D})$ is the image of the unit disk $\mathbb{D}$ under the mappings $q_{0}(z)=1+\sin z$. It is evident from (\ref{star1}) and (\ref{sine}) that
\[\frac{5r}{1-r^{2}} \leq \sin 1-\frac{r^{2}}{1-r^{2}} .\]
Hence the disk in (\ref{star1}) lies inside $\Omega_{s}$ provided $0 < r \leq R_{\mathcal{S}_{sin}^{*}}$. For the function $f_{1}$ defined in (\ref{extremal}) (See Figure \ref{fig:pb1a}.(\subref{fig1:sub4})), at $z:=-\rho= -R_{\mathcal{S}_{sin}^{*}}$,
\[\left|\frac{zf'(z)}{f(z)}\right|=\left|\frac{1+5\rho}{1-\rho^2}\right|=1+\sin 1=q_{0}(1)\in\partial\Omega_{s}(\mathbb{D}).\]
\begin{figure}
	\centering
	\begin{subfigure}{.5\textwidth}
		\centering
		\includegraphics[height=3cm,keepaspectratio]{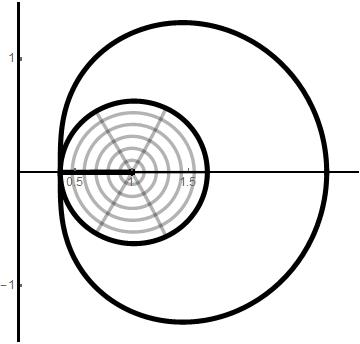}
		\caption{Sharpness of class $\mathcal{S}_{e}^{*}$}
		\label{fig1:sub3}
	\end{subfigure}%
	\begin{subfigure}{.5\textwidth}
		\centering
		\includegraphics[height=3cm,keepaspectratio]{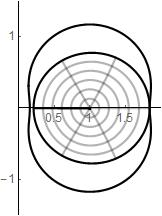}
		\caption{Sharpness of class $\mathcal{S}_{sin}^{*}$}
		\label{fig1:sub4}
	\end{subfigure}
	\caption{Sharpness for starlike functions associated with exponential and sine functions.}
	\label{fig:pb1a}
\end{figure}
\item
Let $\mathcal{S}_{\leftmoon}^{*}=\{f \in \mathcal{S}^{*}:\left|(zf'(z)/f(z))^{2} - 1\right| < 2 \left|zf'(z)/f(z)\right|\}$.
In 2015,  Sok\'{o}\l{} \cite{raina} proved that if $f\in \mathcal{S}_{\leftmoon}^{*}$ then the function $f$ is a starlike function and $\left|zf'(z)/f(z) - 1\right| < \sqrt{2}$ and $\left|zf'(z)/f(z) + 1\right| > \sqrt{2}$.  Interpreting these conditions geometrically, for $f\in \mathcal{S}_{\leftmoon}^{*}$ we see that $w=zf'(z)/f(z)$ lies in the right half plane,  inside the intersection of disks $\left\{w:|w - 1| < \sqrt{2}\right\}$ and $\left\{w:|w + 1| < \sqrt{2}\right\}$.
According to \cite[Lemma 2.1]{gandhi}, we have
\begin{equation}\label{lune}
	\{w \in \mathbb{C}: \left | w-a \right |< 1-|\sqrt{2}-a |\}\subseteq \{w \in \mathbb{C}: \left | w^{2}-1 \right |< 2\left|w\right| \}
\end{equation} or$, (1-\sqrt{2})r^{2}+5r+\sqrt{2}-2 \leq 0$.
If $0 < r \leq R_{\mathcal{S}_{\leftmoon}^{*}}$, then
\[\frac{5r-1}{1-r^{2}} \leq 1-\sqrt{2}.\]
Thus, by \eqref{lune}, the disk in \eqref{star1} lies inside $\{w \in \mathbb{C}: \left | w^{2}-1 \right |< 2\left|w\right| \}$ and hence  $f \in \mathcal{S}_{\leftmoon}^{*}$. The sharpness follows from the functions defined in (\ref{extremal}).
At $z:=\rho= R_{\mathcal{S}_{\leftmoon}^{*}}$, we have
\[ \left| \left(\frac{zf_{1}'(z)}{f_{1}(z)}\right)^{2} - 1    \right|= \left| \left(\frac{1-5\rho}{1-\rho^{2}}\right)^{2} - 1    \right|=2\left| \frac{1-5\rho}{1-\rho^{2}} \right|=2\left| \frac{zf_{1}'(z)}{f_{1}(z)} \right|.\]

\item For $2(\sqrt{2}-1) < a \leq \sqrt{2}$, by \cite[ Lemma 2.2]{sushil},
\begin{equation}\label{rational}
	\{w \in \mathbb{C}: \left | w-a \right |< a-2(\sqrt{2}-1)\}\subseteq \psi(\mathbb{D})
\end{equation}
where $\psi$ is given by $\psi(z)=1+\left(z^2k+z^2/(k^2-kz)\right),\quad k=\sqrt{2}+1$.
 If $0 < r \leq R_{\mathcal{S}_{R}^{*}}$, $2(\sqrt{2}-1) < a= 1/(1-r^{2}) \leq \sqrt{2}$ and
 \[\frac{5r-1}{1-r^{2}} \leq 2-2\sqrt{2},  \quad (0 < r \leq R_{\mathcal{S}_{R}^{*}}).\] Then, by \eqref{rational}, the disk in (\ref{star1}) lies inside $\psi({\mathbb{D}})$. The result is sharp for the function defined in (\ref{extremal}). At $z:=\rho= R_{\mathcal{S}_{R}^{*}}$,
 \[\left|\frac{zf'(z)}{f(z)}\right|=\left|\frac{1-5\rho}{1-\rho^2}\right|=2(\sqrt{2}-1)=\psi(1)\in\partial \psi(\mathbb{D}).\]

 \item For $\sqrt{2}/3 \leq a < \sqrt{2}$, by \cite[Lemma 3.2]{rajni1}, we have
 \begin{equation}\label{reverse}
 	\{w \in \mathbb{C}: \left | w-a \right |< r_{RL}\}\subseteq \{w \in \mathbb{C}:  | (w-\sqrt{2})^{2}-1  |< 1 \},
 \end{equation}
 provided $r_{RL}= \left(\left(1-\left(\sqrt{2}-a\right)^{2}\right)^{1/2}- \left(1-\left(\sqrt{2}-a\right)^{2}\right)\right)^{1/2}$.
 If $0 < r \leq R_{\mathcal{S}_{RL}^{*}}$, then it follows that  $ \sqrt{2}/3 \leq a= 1/(1-r^{2}) < \sqrt{2}$, and
 \begin{align*}
 	 25r^{2}- (1-r^{2})\sqrt{(1-r^{2})^{2}-((\sqrt{2}-\sqrt{2}r^{2})-1)^{2}} & \\   {} +         (1-r^{2})^{2}-((\sqrt{2}-\sqrt{2}r^{2})-1)^{2} & \leq 0.
 \end{align*}
 Then, by \eqref{reverse}, the disk in \eqref{star1} lies inside the region $\{w: | (w-\sqrt{2})^{2}-1 |< 1\}$. The result is sharp for the function defined in (\ref{extremal}).

 \item
 If a function $f(z)$ is strongly starlike of order $\gamma,  0< \gamma \leq 1$,  then $\left|\operatorname{arg}\{zf'(z)/f(z)\}\right| \leq \pi\gamma/2 $.  In other words the values of  $\left|zf'(z)/f(z) \right|$ are in the sector $\left|y\right|\leq \tan(\pi\gamma/2)x$, $x \geq 0$.  In 1997,  Gangadharan \emph{et al.}\ \cite[Lemma 3.1]{ganga} proved that
 \begin{equation}\label{strongstar}
 	\{w \in \mathbb{C}: \left | w-a \right|<   a  \sin(\pi\gamma/2) \} \subseteq\{w: \left| \arg w \right| \leq (\pi\gamma)/2\},  \quad 0<\gamma \leq 1.
 \end{equation}
 If $0 < r \leq R_{\mathcal{S}_\gamma^{*}}$, $0 < \gamma \leq 1$, then
$ r\leq (\sin(\pi\gamma)/2)/5$.
 It is evident from \eqref{strongstar} that the disk in (\ref{star1}) is contained in the sector $\left| \arg w \right| \leq (\pi\gamma)/2,  0<\gamma \leq 1$ if $0 < r \leq R_{\mathcal{S}_\gamma^{*}}$.\qedhere
\end{enumerate}
\end{proof}

\begin{remark}(i) The class \[ \mathcal{S}^{*}(c,d)=\left\{f\in \mathcal{A}: \left|\frac{zf'(z)}{f(z)}-c\right|< d \right\}\] is very closely related to the class of Janowski starlike functions. For the class $\mathcal{{F}}_1$, the $ {\mathcal{S}_{p}}$ radius, $\mathcal{S}^{*}(1/2)$ radius and  $\mathcal{S}^{*}(1,1/2)$ radius  are all equal. It is also clear that  $\mathcal{S}^{*}(1,\sqrt{2}-1) \subseteq \mathcal{S}_{L}^{*}$.  The $\mathcal{S}_{L}^{*}$ radius and $ \mathcal{S}^{*}(1, \sqrt{2}-1)$ radius are also equal.

(ii) The radius of strong starlikeness is clearly sharp for $\gamma=0$ and for other cases, we have given a lower bound only. 
\end{remark}

\begin{definition}
	Let $\mathcal{F}_{2}$ be the class of functions $f \in \mathcal{A}$ satisfying the inequality \[\left|\frac{f(z)}{g(z)}-1\right|<1 \quad  (z \in \mathbb{D}) \]
	for some $g \in \mathcal{A}$ with \[ \operatorname{Re}\left (\frac{1+z}{z}g(z)\right )>0  \quad (z \in \mathbb{D}).\]
\end{definition}
The functions $ f_2,  g_2: \mathbb{D} \longrightarrow \mathbb{C}$ defined by
\begin{equation}\label{extremal1}
f_{2}(z)=\frac{z(1-z)^{2}}{(1+z)^{2}}   \quad\text{and}\quad   g_{2}(z)=\frac{z(1-z)}{(1+z)^{2}}
\end{equation}
satisfy
 \[\left| \frac{f_{2}(z)}{g_{2}(z)}-1 \right| =|z| < 1, \quad
  \operatorname{Re}\frac{(1+z)}{z}g_{2}(z)=\operatorname{Re}\frac{1-z}{1+z}>0\] and hence the function $f_2\in \mathcal{F}_{2}$. This proves that the class $ \mathcal{F}_{2}$ is non-empty and the function $f_2$ is extremal function for the radius problems we consider.  As
\[ f_2(z)=z-4 z^2+8 z^3-12 z^4+\dotsc,\]
the functions in  $ \mathcal{F}_{2}$ are not necessarily  univalent.
Since
\[ f_2'(z)=\frac{ 1-5 z+3 z^2+ z^3 }{(1+z)^3},\]  we have $f_2'(\sqrt{5}-2)=0$ and
it follows by the first part of the following theorem,  that the radius of univalence of the functions in class $\mathcal{F}_{2}$ is $\sqrt{5}-2 \approx 0.2361$.

	\begin{theorem}
		For the class $\mathcal{F}_{2}$  the following results hold:
		\begin{enumerate}[label=(\roman*)]
			\item The $\mathcal{S}^{*}(\alpha)$ radius
			$R_{\mathcal{S}^{*}(\alpha)}=(1-\alpha)/\left(2+\sqrt{4+(\alpha-1)^{2}}\right)$$,\quad 0 \leq \alpha< 1$.	
			
			\item The $\mathcal{S}_{L}^{*}$ radius
			$R_{\mathcal{S}_{L}^{*}} \geq (\sqrt{5}-2)/(1+\sqrt{2}) \approx 0.0977$.
			
			\item The $\mathcal{S}_{p}$ radius
			$R_{\mathcal{S}_{p}}= \sqrt{17}-4 \approx 0.1231$.
			
			\item The $\mathcal{S}_{ e}^{*}$ radius
			$R_{\mathcal{S}_{e}^{*}}= (2e-2)/(4e+\sqrt{20 e^{2}-8 e+4 }) \approx 0.1543$.
			
			\item The $\mathcal{S}_{c}^{*}$ radius
			$R_{\mathcal{S}_{c}^{*}}=  \sqrt{10}-3\approx 0.1623$.

			\item The $\mathcal{S}_{sin}^{*}$ radius
			$R_{\mathcal{S}_{sin}^{*}} \geq \sin 1 /\left(2+\sqrt{4+ \sin 1 (2+ \sin 1)}\right) \approx 0.1858$.
			
			\item The $\mathcal{S}_{\leftmoon}^{*}$ radius
			$R_{\mathcal{S}_{\leftmoon}^{*}}= (2-\sqrt{2})/\left(2+\sqrt{10-4\sqrt{2}}\right)\approx 0.1434$.
			
			\item The $\mathcal{S}_{R}^{*}$ radius
			$R_{\mathcal{S}_{R}^{*}}= (3-2\sqrt{2})/(2+\sqrt{21-12\sqrt{2}}) \approx 0.0428$.
			
			\item The $\mathcal{S}_{RL}^{*}$ radius
			$R_{\mathcal{S}_{RL}^{*}} $ is atleast the smallest root $(\approx 0.0692)$ in $[0,1]$ of the equation
			\begin{align*}
			r^2(r+4)^2+(r^2-1)^2+(r^2-1)\sqrt{(r^2+\sqrt{2})(2-\sqrt{2}-r^2)}&\\   -(1+\sqrt{2}(r^2-1))^2      &= 0 .
			\end{align*}
			
			\item The $\mathcal{S}_{\gamma}^{*}$ radius  $R_{\mathcal{S}_\gamma^{*}} \geq \sin(\pi\gamma/2)/(2+\sqrt{4+\sin(\pi\gamma/2)},  \quad 0 < \gamma \leq 1$.

		\end{enumerate}
		\end{theorem}

	\begin{proof}
		Let the function $ f \in \mathcal{F}_{2}$. Then $\left|f(z)/g(z)-1\right|<1$ if and only if $	\operatorname{Re} (g(z)/f(z))>1/2 $.  Let the function $ g: \mathbb{D} \longrightarrow \mathbb{C}$ be chosen such that
		\begin {equation} \label{g1}
		\operatorname{Re}\frac{g(z)}{f(z)}>1/2 \quad \text{and} \quad
		\operatorname{Re}\left (\frac{1+z}{z}g(z)\right )>0  \quad (z \in \mathbb{D}).
	\end{equation}
Define the functions $p_{1}, p_{2}:\mathbb{D} \longrightarrow \mathbb{C} $ by
\begin{equation} \label{g0}
	p_{1}(z)=\frac{1+z}{z}g(z) \quad \text{and} \quad p_{2}(z)= \frac{g(z)}{f(z)}.
\end{equation}
By (\ref{g1}) and (\ref{g0}),  we have $ p_{1} \in \mathcal{P}$, $ p_{2} \in \mathcal{P}(1/2)$
and $f(z)=z p_{1}(z)/((1+z) p_{2}(z))$.
It can be shown by calculation that
\begin{equation} \label{main2}
\frac{zf'(z)}{f(z)}=\frac{z p_{1}^{'}(z)}{p_{1}(z)}-\frac{z p_{2}^{'}(z)}{p_{2}(z)}+\frac{1}{1+z}.
\end{equation}
Using (\ref{disk}) and (\ref{shah}), it follows from (\ref{main2}) that the function $f$ maps the disk $\left|z\right| \leq r$ onto the disk
\begin{equation} \label{star2}
\left | \frac{zf'(z)}{f(z)} - \frac{1}{1-r^{2}}  \right |  \leq \frac{4r+r^{2}}{1-r^{2}}.
\end{equation}The classes we discuss here are all subclasses of starlike functions. By \eqref{star2}, we have
\begin{equation}\label{starlike2}
	\operatorname{Re} \frac{zf'(z)}{f(z)} \geq \frac{1-4r-r^{2}}{1-r^{2}} \geq 0, \quad (r \leq \sqrt{5}-2).
\end{equation}
Hence all the radii that we estimate will be less than $\sqrt{5}-2 \approx 0.2361$. Also,  for $0 < r \leq \sqrt{5}-2$,  the centre of disk in (\ref{star2}) lies in the interval $[1,  1/4(\sqrt{5}-2)] \approx [1, 1.05902]$.

\begin{enumerate}[label=(\roman*)]
	
	\item
	The number $r=R_{\mathcal{S}^{*}(\alpha)} $ is the root of
	$  (\alpha-1) r^{2}-4r+ 1-\alpha=0$ in $[0,1]$ and hence, for $0<r\leq R_{\mathcal{S}^{*}(\alpha)}$, it follows
	by (\ref{starlike2}) that
	\[ 	\operatorname{Re} \frac{zf'(z)}{f(z)} \geq   \frac{1-4r-r^{2}}{1-r^{2}}\geq \alpha.\]
	For the functions  $f_{2} \in \mathcal{F}_{2}$  given by \eqref{extremal1}, we have
	\[\frac{zf_{2}^{'}(z)}{f_{2}(z)}=\frac{1-4z-z^2}{1-z^{2}}= \frac{1-4r-r^{2}}{1-r^{2}}=\alpha, \quad  (z=r= R_{\mathcal{S}^{*}(\alpha)})\]
	and this shows that the radius is sharp.	

	\item
	It   follows from (\ref{star2}) that
	\begin{equation} \label{estimate1}
	\left | \frac{zf'(z)}{f(z)} - 1  \right |\leq \left | \frac{zf'(z)}{f(z)} - \frac{1}{1-r^{2}}  \right |+\frac{r^2}{1-r^2}  \leq \frac{4r+ 2r^{2}}{1-r^{2}}.
	\end{equation} The number $r=R_{\mathcal{S}_{L}^{*}}$ is the root in $[0,1]$ of $(4r+ 2r^{2})=\sqrt{2}-1(1-r^{2})$ and for $0 < r \leq R_{\mathcal{S}_{L}^{*}}$,  we have
	\begin{equation} \label{lemniradius1}
	\frac{4r+ 2r^{2}}{1-r^{2}} \leq \sqrt{2}-1.
	\end{equation}
	Therefore,  by (\ref{estimate1}) and (\ref{lemniradius1}), for $0 < r \leq R_{\mathcal{S}_{L}^{*}}$, we have
	\begin{equation}\label{connect2}
		\left | \frac{zf'(z)}{f(z)} - 1  \right |  \leq \frac{4r+ 2r^{2}}{1-r^{2}} \leq \sqrt{2}-1.
	\end{equation}
	For $0 < r \leq R_{\mathcal{S}_{L}^{*}}$,  using triangle inequality and \eqref{connect2}, we have
	\begin{equation}\label{tconnect2}
	\left | \frac{zf'(z)}{f(z)} + 1  \right | \leq \sqrt{2}+1	
	\end{equation}
	 and hence by \eqref{connect2} and \eqref{tconnect2}
	\[ \left|\left(\frac{zf'(z)}{f(z)}\right)^{2} - 1\right|\leq \left|\frac{zf'(z)}{f(z)} + 1\right|\left|\frac{zf'(z)}{f(z)} - 1\right|\leq (\sqrt{2}+1)(\sqrt{2}-1)=1.\]The result obtained is not sharp.
		
	\item
	If $0 < r \leq R_{\mathcal{S}_{p}}$,  then $a= 1/(1-r^{2}) \leq 3/2$ and
	\[\frac{4r+r^{2}}{1-r^{2}} \leq \frac{1}{1-r^{2}}-\frac{1}{2}.\] Thus,  by \eqref{parabolic},  we see that the disk in (\ref{star2}) lies inside the parabolic region $\Omega_{PAR}$.
	Sharpness follows for the function $f_{2}$ defined in (\ref{extremal1}). At $z:=\rho= R_{\mathcal{S}_{p}}$, we have
	\[\operatorname{Re}\frac{zf'(z)}{f(z)}=\frac{1-4\rho-\rho^2}{1-\rho^2}=\frac{4\rho}{1-\rho^2}=\left|\frac{-4\rho}{1-\rho^2}\right|=\left|\frac{zf'(z)}{f(z)}-1\right|.\]

\item
 For $0 < r \leq R_{\mathcal{S}_{e}^{*}}$, we have
$1/e \leq  a= 1/(1-r^{2}) \leq (\ e+\ e^{-1})/2$ and
\[\frac{4r+r^{2}}{1-r^{2}} \leq \frac{1}{1-r^{2}}-\frac{1}{e}.\]
By  \eqref{exponential}, the disk in \eqref{star2}  lies inside
$ \Omega_{e} $   for $0 < r \leq R_{\mathcal{S}_{e}^{*}}$ proving   that  the $\mathcal{S}_{e}^{*}$ radius for the class $\mathcal{F}_{2}$ is $R_{\mathcal{S}_{e}^{*}}$.
The sharpness follows for the function $f_{2}$ defined in (\ref{extremal1}).
Indeed at $z:=\rho= R_{\mathcal{S}_{e}^{*}}$, we have  \[ \left| \log \frac{zf_{2}'(z)}{f_{2}(z)} \right|= \left| \log \frac{1-4 \rho-\rho^{2}}{1-\rho^{2}}\right|=1.\]

\item
If $0 < r \leq R_{\mathcal{S}_{c}^{*}}$, then
\[\frac{4r+r^{2}}{1-r^{2}} \leq  \frac{1}{1-r^{2}}-\frac{1}{3} .\]  By (\ref{cardiod}), we see that the disk in (\ref{star2}) lies inside $\Omega_{c}$, if $0 < r \leq R_{\mathcal{S}_{c}^{*}}$. The  result is sharp for the function $f_2$ defined in (\ref{extremal1}) (See Figure \ref{fig:pb2}.(\subref{fig2:sub1})).  At $z:=\rho= R_{\mathcal{S}_{c}^{*}}$,
\[\left|\frac{zf'(z)}{f(z)}\right|=\left|\frac{1-4\rho-\rho^2}{1-\rho^2}\right|=\frac{1}{3}=\Omega_c(-1)\in\partial \Omega_c(\mathbb{D}).\]
\item
For $0 < r \leq R_{\mathcal{S}_{sin}^{*}}$, $\left|a-1\right| \leq \sin 1$ and \[\frac{4r+r^{2}}{1-r^{2}} \leq  \sin 1-\frac{r^{2}}{1-r^{2}} ,\] then
it is evident from (\ref{sine}) that the disk in (\ref{star2}) lies inside $\Omega_{s}$.
The radius is not sharp.
\item
If $0 < r \leq R_{\mathcal{S}_{\leftmoon}^{*}}$, then
\[\frac{r^{2}+4r-1}{1-r^{2}} \leq 1-\sqrt{2}.\]
Thus by \eqref{lune}, the disk in \eqref{star2} lies inside $\{w \in \mathbb{C}: \left | w^{2}-1 \right |< 2\left|w\right| \}$ and hence $ f \in \mathcal{S}_{\leftmoon}^{*}$.
The sharpness follows for the functions defined in (\ref{extremal1}) (See Figure \ref{fig:pb2}.(\subref{fig2:sub2})).
At $z:=\rho= R_{\mathcal{S}_{\leftmoon}^{*}}$, we have
\[ \left| \left(\frac{zf_{2}'(z)}{f_{2}(z)}\right)^{2} - 1    \right|= \left| \left(\frac{\rho^{2}+4\rho-1}{1-\rho^{2}}\right)^{2} - 1    \right|=2\left| \frac{\rho^{2}+4\rho-1}{1-\rho^{2}} \right|=2\left| \frac{zf_{1}'(z)}{f_{1}(z)}\right|.\]
\begin{figure}
	\centering
	\begin{subfigure}{.5\textwidth}
		\centering
		\includegraphics[height=3cm,keepaspectratio]{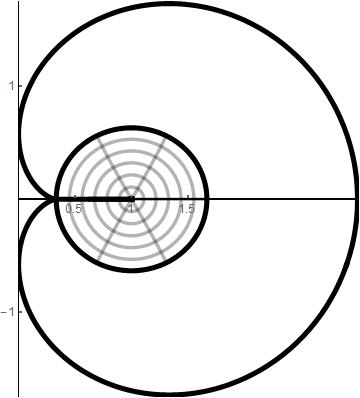}
		\caption{Sharpness of class $\mathcal{S}_{c}^{*}$}
		\label{fig2:sub1}
	\end{subfigure}%
	\begin{subfigure}{.5\textwidth}
		\centering
		\includegraphics[height=3cm,keepaspectratio]{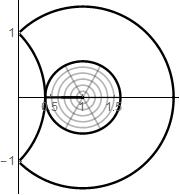}
		\caption{Sharpness of class $\mathcal{S}_{\leftmoon}^{*}$}
		\label{fig2:sub2}
	\end{subfigure}
	\caption{Sharpness of starlike functions associated with cardioid and lune.}
	\label{fig:pb2}
\end{figure}
\item
If $0 < r \leq R_{\mathcal{S}_{R}^{*}}$, $2(\sqrt{2}-1) < a= 1/(1-r^{2}) \leq \sqrt{2}$ and
\[\frac{r^{2}+4r-1}{1-r^{2}} \leq 2-2\sqrt{2},  \quad 0 < r \leq R_{\mathcal{S}_{R}^{*}}.\] Then, by \eqref{rational}, the disk in (\ref{star2}) lies inside $\psi(\mathbb{D})$. For the function defined in (\ref{extremal1}),  at $z:=\rho= R_{\mathcal{S}_{R}^{*}}$,
\[\left|\frac{zf'(z)}{f(z)}\right|=\left|\frac{1-4\rho-\rho^2}{1-\rho^2}\right|=2(\sqrt{2}-1)=\psi(1)\in\partial \psi(\mathbb{D}).\]

\item
If $0 < r \leq R_{\mathcal{S}_{RL}^{*}}$, $ \sqrt{2}/3 \leq a= 1/(1-r^{2}) < \sqrt{2}$, and
\begin{align*}
(4r+r^{2})^{2}- (1-r^{2})\sqrt{(1-r^{2})^{2}-((\sqrt{2}-\sqrt{2}r^{2})-1)^{2}}&\\ +         (1-r^{2})^{2}-((\sqrt{2}-\sqrt{2}r^{2})-1)^{2}&\leq 0.
\end{align*}
Then by \eqref{reverse} the disk in \eqref{star2} lies inside the region $\{w: | (w-\sqrt{2})^{2}-1 |< 1\}$. The result obtained is not sharp.

\item
If $0 < r \leq R_{\mathcal{S}_\gamma^{*}}$ and $0 < \gamma \leq 1$, then
$ r^{2}+4r\leq (\sin(\pi\gamma)/2)$.
It is evident from \eqref{strongstar} that the disk in (\ref{star2}) is contained in the sector $\left| \arg w \right| \leq (\pi\gamma)/2,  0<\gamma \leq 1$ if $0 < r \leq R_{\mathcal{S}_\gamma^{*}}$. \qedhere
\end{enumerate}
\end{proof}
The problem of determination of the exact radii of starlikeness associated with lemniscate of Bernoulli,  sine function, reverse lemniscate and strongly starlike function is open.


\begin{definition}
	Let $\mathcal{{F}}_{3}$ be the class of functions $f \in \mathcal{A}$ satisfying the inequality  \[ \operatorname{Re}\left (\frac{1+z}{z}f(z)\right )>0  \quad (z \in \mathbb{D}).\]
\end{definition}
The function $ f_3: \mathbb{D} \longrightarrow \mathbb{C}$ defined by
\begin{equation}\label{extremal3}
f_{3}(z)=\frac{z(1-z)}{(1+z)^{2}}
\end{equation}
satisfy \[\operatorname{Re}\frac{(1+z)f_{3}(z)}{z}=\operatorname{Re}\frac{1-z}{1+z}>0\]
 and hence the function $f_3\in \mathcal{{F}}_{3}$. This proves that the class $ \mathcal{{F}}_{3}$ is non-empty. This function $f_3$ is extremal function for the radii problem we consider. As
\[ f_3(z)=z-3 z^2+5 z^3-7 z^4+\dotsc, \]
the functions in  $ \mathcal{{F}}_{3}$ are not necessarily  univalent.
Since
\[ f_3'(z)=\frac{ 1-3z }{(1+z)^3},\]  we have $f_3'(1/3)=0$ and
it follows, by the first part of the following theorem,  that the radius of univalence of the functions in class $\mathcal{{F}}_{3}$ is $1/3 \approx 0.3333$.

\begin{theorem}
	For the class $\mathcal{{F}}_{3}$,  the following results hold:
	\begin{enumerate}[label=(\roman*)]
		\item The $\mathcal{S}^{*}(\alpha)$ radius
		$R_{\mathcal{S}^{*}(\alpha)}=2(1-\alpha)/\left(3+\sqrt{9-4\alpha(1-\alpha)}\right)$$, \quad 0 \leq \alpha< 1$.	
		\item The $\mathcal{S}_{L}^{*}$ radius
		$R_{\mathcal{S}_{L}^{*}}=(2\sqrt{2}-2)/\left(3+\sqrt{9-4\sqrt{2}(1-\sqrt{2})}\right) \approx 0.1301 $.
		
		\item The $\mathcal{S}_{p}$ radius
		$R_{\mathcal{S}_{p}}= 3-2\sqrt{2} \approx 0.1716$.
		
		\item The $\mathcal{S}_{e}^{*}$ radius
		$R_{\mathcal{S}_{e}^{*}}= (2e-2)/\left(3e+\sqrt{9e^{2}+4(1-e) } \right)\approx 0.2165$.
		
		\item The $\mathcal{S}_{c}^{*}$ radius
		$R_{\mathcal{S}_{c}^{*}}= (9-\sqrt{73})/2 \approx 0.2279$.

		\item The $\mathcal{S}_{sin}^{*}$ radius
		$R_{\mathcal{S}_{sin}^{*}}= 2 \sin 1/\left(3+\sqrt{9+4 \sin 1 (1+ \sin 1)}\right) \approx 0.2439$.
		
		\item The $\mathcal{S}_{\leftmoon}^{*}$ radius
		$R_{\mathcal{S}_{\leftmoon}^{*}}= (4-2\sqrt{2})/\left(3+\sqrt{25-12\sqrt{2}}\right) \approx 0.2008$.
		
		\item The $\mathcal{S}_{R}^{*}$ radius
		$R_{\mathcal{S}_{R}^{*}}=  (6-4\sqrt{2})/\left(3+\sqrt{65-40\sqrt{2}}\right) \approx 0.0581$.
		
		\item The $\mathcal{S}_{RL}^{*}$ radius
		$R_{\mathcal{S}_{RL}^{*}} $ is the root $(\approx 0.0926)$ in $[0,1]$ of the equation
	\begin{align*}
	9r^{2}+(r^2-1)^{2}+(r^2-1)\sqrt{(r^2+\sqrt{2})(2-\sqrt{2}-r^2)}	&\\  -(1+ \sqrt{2}(r^2-1))^{2}      &= 0.
	\end{align*}
		
		\item The $\mathcal{S}_\gamma^{*}$ radius  $R_{\mathcal{S}_\gamma^{*}} \geq \sin(\pi\gamma/2)/3,  \quad 0 < \gamma \leq 1$.
		
	\end{enumerate}
\end{theorem}
\begin{proof}
	
	Let the function $ f \in \mathcal{{F}}_{3}$. Then
	\begin{equation} \label{g2}
	\operatorname{Re}\left (\frac{1+z}{z}f(z)\right )>0  \quad (z \in \mathbb{D}).
	\end{equation}
	Define the function $h:\mathbb{D} \longrightarrow \mathbb{C} $ by
	\begin{equation}\label{g5}
	h(z)=\frac{1+z}{z}f(z).
	\end{equation}
	By (\ref{g2}) and (\ref{g5}) we have $h \in \mathcal{P}$
	and $f(z)= z h(z)/(1+z) $.  \\
	Therefore,  by calculation it can be shown that
	\begin{equation} \label{main3}
	\frac{zf'(z)}{f(z)}=\frac{z h^{'}(z)}{h(z)}+\frac{1}{1+z}.
	\end{equation}
	Using (\ref{disk}) and (\ref{shah}), it follows from (\ref{main3}) that $f$ maps the disk $\left|z\right| \leq r$ onto the disk
	\begin{equation} \label{star3}
	\left | \frac{zf'(z)}{f(z)} - \frac{1}{1-r^{2}}  \right |  \leq \frac{3r}{1-r^{2}}.
	\end{equation}As the classes we discuss here are all subclasses of starlike functions. By \eqref{star3}, we have
	\begin{equation}\label{starlike3}
		\operatorname{Re} \frac{zf'(z)}{f(z)} \geq \frac{1-3r}{1-r^{2}} \geq 0, \quad (r \leq 1/3).
	\end{equation}
	Hence all the radii that we estimate will be less than $1/3$. Note that, for $0 < r \leq 1/3$,  the centre of disk in (\ref{star3}) lies in the interval $[1,  9/8] \approx [1, 1.125]$.
	
	\begin{enumerate}[label=(\roman*)]
		\item
		The number $r=R_{\mathcal{S}^{*}(\alpha)} $ is the root of
		$  \alpha r^{2}-3r+ 1-\alpha=0$ in $[0,1]$ and hence, for $0<r\leq R_{\mathcal{S}^{*}(\alpha)}$, it follows
		by (\ref{starlike3}) that
		\[ 	\operatorname{Re} \frac{zf'(z)}{f(z)} \geq   \frac{1-3r}{1-r^{2}}\geq \alpha.\]
		For the function  $f_{3} \in \mathcal{{F}}_{3}$  given by \eqref{extremal3}, we have
		\[\frac{zf_{3}^{'}(z)}{f_{3}(z)}=\frac{1-3z}{1-z^{2}}= \frac{1-3r}{1-r^{2}}=\alpha, \quad (z=r= R_{\mathcal{S}^{*}(\alpha)}) \]
		and this shows that the radius is sharp (See Figure \ref{fig:pb3}.(\subref{fig3:sub1})).
	
		\item
		It   follows from (\ref{star3}) that
		\begin{equation} \label{estimate3}
		\left | \frac{zf'(z)}{f(z)} - 1  \right |\leq \left| \frac{zf'(z)}{f(z)} - \frac{1}{1-r^2}  \right|+\frac{r^2}{1-r^2} \leq \frac{3r+ r^{2}}{1-r^{2}}.
		\end{equation}The number $r=R_{\mathcal{S}_{L}^{*}}$ is the root in $[0,1]$ of $(3r+ r^{2})=(\sqrt{2}-1)(1-r^{2})$ and for $0 < r \leq R_{\mathcal{S}_{L}^{*}}$,  we have
		\begin{equation} \label{lemniradius3}
		\frac{3r+r^{2}}{1-r^{2}} \leq \sqrt{2}-1.
		\end{equation}
		Therefore,  by (\ref{estimate3}) and (\ref{lemniradius3}), for $0 < r \leq R_{\mathcal{S}_{L}^{*}}$, we have
		\begin{equation}\label{connect3}
		\left | \frac{zf'(z)}{f(z)} - 1  \right |  \leq \frac{3r+ r^{2}}{1-r^{2}} \leq \sqrt{2}-1.
		\end{equation}
		For $0 < r \leq R_{\mathcal{S}_{L}^{*}}$,  using triangle inequality and \eqref{connect3}, we have
		\begin{equation}\label{tconnect3}
			\left | \frac{zf'(z)}{f(z)} + 1  \right | \leq \sqrt{2}+1
		\end{equation}
		and hence by \eqref{connect3} and \eqref{tconnect3}
		\[ \left|\left(\frac{zf'(z)}{f(z)}\right)^{2} - 1\right|\leq \left|\frac{zf'(z)}{f(z)} + 1\right|\left|\frac{zf'(z)}{f(z)} - 1\right|\leq (\sqrt{2}+1)(\sqrt{2}-1)=1.\]
		The number  $\rho=  R_{\mathcal{S}_{L}^{*}}$ satisfies $(1+3\rho)/(1-\rho^{2})= \sqrt{2}$. Using this, we see that the function $f_{3}$ defined in (\ref{extremal3}) satisfies
		\[ \left| \left(\frac{zf_{3}'(z)}{f_{3}(z)}\right)^{2} - 1    \right|=   \left| \left(\frac{1-3z}{1-z^{2}}\right )^{2}-1\right|= \left| \left(\frac{1+3\rho}{1-\rho^{2}}\right)^{2} - 1    \right|=1, \quad (z:= -\rho= -R_{\mathcal{S}_{L}^{*}}).\]
		This shows that the radius is sharp.
		
		\item
		If $0 < r \leq R_{\mathcal{S}_{p}}$,  then $a= 1/(1-r^{2}) \leq 3/2$ and
		\[\frac{3r}{1-r^{2}} \leq \frac{1}{1-r^{2}}-\frac{1}{2}.\] Thus,  by \eqref{parabolic},  we see that the disk in (\ref{star3}) lies inside the parabolic region $\Omega_{PAR}$.  Sharpness follows for the function $f_{3}$ defined in (\ref{extremal3}). At $z:=\rho= R_{\mathcal{S}_{p}}$, we have
		\[\operatorname{Re}\frac{zf'(z)}{f(z)}=\frac{1-3\rho}{1-\rho^2}=\frac{3\rho-\rho^2}{1-\rho^2}=\left|\frac{\rho^2-3\rho}{1-\rho^2}\right|=\left|\frac{zf'(z)}{f(z)}-1\right|.\]
		
		\item
		For $0 < r \leq R_{\mathcal{S}_{e}^{*}}$, we have  $1/e \leq  a= 1/(1-r^{2}) \leq (\ e+\ e^{-1})/2$ and
		\[\frac{3r}{1-r^{2}} \leq \frac{1}{1-r^{2}}-\frac{1}{e}.\]
		By  \eqref{exponential}, the disk in \eqref{star3}  lies inside
		$ \Omega_{e} $   for $0 < r \leq R_{\mathcal{S}_{e}^{*}}$ and it proves   that  the $\mathcal{S}_{e}^{*}$ radius for the class $\mathcal{F}_{3}$ is $R_{\mathcal{S}_{e}^{*}}$.
		The sharpness follows for the function $f_{3}$ defined in (\ref{extremal3}).
		Indeed at $z:=\rho= R_{\mathcal{S}_{e}^{*}}$, we have  \[ \left| \log \frac{zf_{1}'(z)}{f_{1}(z)} \right|= \left| \log \frac{1-3 \rho}{1-\rho^{2}}\right|=1.\]
		
		\item
		If $0 < r \leq R_{\mathcal{S}_{c}^{*}}$, then
		\[\frac{3r}{1-r^{2}} \leq \frac{1}{1-r^{2}}-\frac{1}{3} .\]  By (\ref{cardiod}), we see that the disk in (\ref{star3}) lies inside $\Omega_{c}$,  if $0 < r \leq R_{\mathcal{S}_{c}^{*}}$. The  result is sharp for the function $f_3$ defined in (\ref{extremal3}) (See Figure \ref{fig:pb3}.(\subref{fig3:sub2})). At $z:=\rho= R_{\mathcal{S}_{c}^{*}}$,
		\[\left|\frac{zf'(z)}{f(z)}\right|=\left|\frac{1-3\rho}{1-\rho^2}\right|=\frac{1}{3}=\Omega_c(-1)\in\partial \Omega_c(\mathbb{D}).\]
		\begin{figure}
			\centering
			\begin{subfigure}{.5\textwidth}
				\centering
				\includegraphics[height=3cm,keepaspectratio]{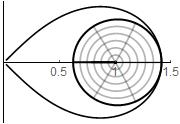}
				\caption{Sharpness of class $\mathcal{S}_{L}^{*}$}
				\label{fig3:sub1}
			\end{subfigure}%
			\begin{subfigure}{.5\textwidth}
				\centering
				\includegraphics[height=3cm,keepaspectratio]{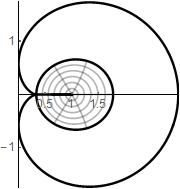}
				\caption{Sharpness of class $\mathcal{S}_{c}^{*}$}
				\label{fig3:sub2}
			\end{subfigure}
			\caption{Sharpness of starlike functions associated with lemniscate and cardioid.}
			\label{fig:pb3}
		\end{figure}
		
		\item
		For $0 < r \leq R_{\mathcal{S}_{sin}^{*}}$
		\[\frac{3r}{1-r^{2}} \leq  \sin 1-\frac{r^{2}}{1-r^{2}} .\] It is evident from (\ref{sine}) that the disk in (\ref{star3}) lies inside $\Omega_{s}$. For the function $f_{3}$ defined in (\ref{extremal3}), at $z:=-\rho=- R_{\mathcal{S}_{sin}^{*}}$,
		\[\left|\frac{zf'(z)}{f(z)}\right|=\left|\frac{1+3\rho}{1-\rho^2}\right|=1+\sin 1=q_0(1)\in\partial \Omega_s(\mathbb{D}).\]
		
		\item
		If $0 < r \leq R_{\mathcal{S}_{\leftmoon}^{*}}$, then
		\[\frac{3r-1}{1-r^{2}} \leq 1-\sqrt{2}.\]
		Thus, by \eqref{lune}, the disk in \eqref{star3} lies inside $\{w \in \mathbb{C}: \left | w^{2}-1 \right |< 2\left|w\right| \}$ and hence $ f \in \mathcal{S}_{\leftmoon}^{*}$.
		The sharpness follows for the function defined in (\ref{extremal3}) (See Figure \ref{fig:pb3a}.(\subref{fig3:sub3})).
		At $z:=\rho= R_{\mathcal{S}_{\leftmoon}^{*}}$, we have
		\[ \left| \left(\frac{zf_{3}'(z)}{f_{3}(z)}\right)^{2} - 1    \right|= \left| \left(\frac{1-3\rho}{1-\rho^{2}}\right)^{2} - 1    \right|=2\left| \frac{1-3\rho}{1-\rho^{2}}\right|=2\left|\frac{zf_{3}'(z)}{f_{3}(z)}\right|.\]
		
		\item
		If $0 < r \leq R_{\mathcal{S}_{R}^{*}}$, $2(\sqrt{2}-1) < a= 1/(1-r^{2}) \leq \sqrt{2}$ and
		\[\frac{3r-1}{1-r^{2}} \leq 2-2\sqrt{2},  \quad 0 < r \leq R_{\mathcal{S}_{R}^{*}}.\] Then, by \eqref{rational}, the disk in (\ref{star3}) lies inside $\psi(\mathbb{D})$. The result is sharp for the function defined in (\ref{extremal3}).  At $z:=\rho= R_{\mathcal{S}_{R}^{*}}$,
		\[\left|\frac{zf'(z)}{f(z)}\right|=\left|\frac{1-3\rho}{1-\rho^2}\right|=2(\sqrt{2}-1)=\psi(1)\in\partial \psi(\mathbb{D}).\]
		
		\item
		If $0 < r \leq R_{\mathcal{S}_{RL}^{*}}$, $ \sqrt{2}/3 \leq a= 1/(1-r^{2}) < \sqrt{2}$, and
		\begin{align*}
		9r^{2}- (1-r^{2})\sqrt{(1-r^{2})^{2}-((\sqrt{2}-\sqrt{2}r^{2})-1)^{2}}&\\ +         (1-r^{2})^{2}-((\sqrt{2}-\sqrt{2}r^{2})-1)^{2}&\leq 0.
		\end{align*}
		Then, by \eqref{reverse}, the disk in \eqref{star3} lies inside the region $\{w: | (w-\sqrt{2})^{2}-1 |< 1\}$. The result is sharp for the function defined in (\ref{extremal3}) (See Figure \ref{fig:pb3a}.(\subref{fig3:sub4})).
			\begin{figure}
			\centering
			\begin{subfigure}{.5\textwidth}
				\centering
				\includegraphics[height=3cm,keepaspectratio]{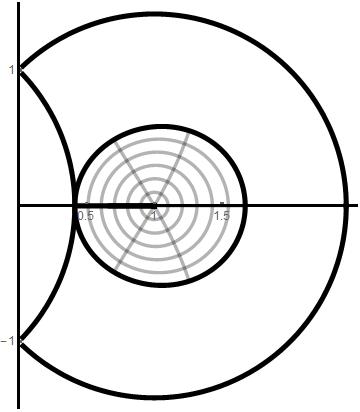}
				\caption{Sharpness of class $\mathcal{S}_{\leftmoon}^{*}$}
				\label{fig3:sub3}
			\end{subfigure}%
			\begin{subfigure}{.5\textwidth}
				\centering
				\includegraphics[height=3cm,keepaspectratio]{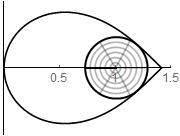}
				\caption{Sharpness of class $\mathcal{S}_{RL}^{*}$}
				\label{fig3:sub4}
			\end{subfigure}
			\caption{Sharpness of starlike functions associated with lune and reverse lemniscate.}
			\label{fig:pb3a}
		\end{figure}
		\item
		 If $0 < r \leq R_{\mathcal{S}_\gamma^{*}}$, $0 < \gamma \leq 1$, then
		$ r\leq (\sin(\pi\gamma)/2)/3$.
		It is evident from \eqref{strongstar} that the disk in (\ref{star3}) is contained in the sector $\left| \arg w \right| \leq (\pi\gamma)/2,  0<\gamma \leq 1$ if $0 < r \leq R_{\mathcal{S}_\gamma^{*}}$. \qedhere
		
	\end{enumerate}
\end{proof}


\begin{definition}
	Let $\mathcal{F}_{4}$ be the class of functions $f \in \mathcal{A}$ satisfying the inequality  \[ \operatorname{Re}\left (\frac{(1+z)^{2}}{z}f(z)\right )>0  \quad (z \in \mathbb{D}).\]
\end{definition}
The functions $ f_4: \mathbb{D} \longrightarrow \mathbb{C}$ defined by
\begin{equation}\label{extremal4}
 f_{4}(z)=\frac{z(1-z)}{(1+z)^{3}}.
\end{equation}
satisfy \[\operatorname{Re}\frac{(1+z)}{z}f_{4}(z)=\operatorname{Re}\frac{1-z}{1+z}>0\]
 and hence the function $f_4\in \mathcal{F}_{4}$. This proves that the class $ \mathcal{F}_{4}$ is non-empty. Also, the function $f_4$ is extremal function for the radii problem we consider.  Since
\[ f_4(z)=z-4 z^2+9 z^3-16 z^4+\dotsc, \]
the functions in  $ \mathcal{F}_{4}$ are not necessarily  univalent.
Since
\[ f_4'(z)=\frac{1-4z+z^2}{(1+z)^4}, \]  we have $f_4'(2-\sqrt{3})=0$ and
it follows by the first part of the following theorem,  that the radius of univalence of the functions in class $\mathcal{F}_{4}$ is $2-\sqrt{3} \approx 0.276949$.

\begin{theorem}
	For the class $\mathcal{F}_{4}$,  the following results hold:
	\begin{enumerate}[label=(\roman*)]
		
		\item The $\mathcal{S}^{*}(\alpha)$ radius
		$R_{\mathcal{S}^{*}(\alpha)}=(1-\alpha)/\left(2+\sqrt{3+\alpha^{2}}\right)$$, \quad 0 \leq \alpha< 1$.	
		\item The $\mathcal{S}_{L}^{*}$ radius
		$R_{\mathcal{S}_{L}^{*}}=(\sqrt{5}-2)/(1+\sqrt{2}) \approx 0.9778 $.
		
		\item The $\mathcal{S}_{p}$ radius
		$R_{\mathcal{S}_{p}}= (4-\sqrt{13})/3 \approx 0.1315$.
		
		\item The $\mathcal{S}_{e}^{*}$ radius
		$R_{\mathcal{S}_{e}^{*}}= (e-1)/\left(2e+\sqrt{3e^{2}+1 }\right)  \approx 0.1676$.
		
		\item The $\mathcal{S}_{c}^{*}$ radius
		$R_{\mathcal{S}_{c}^{*}}= (3 -\sqrt{7})/2 \approx 0.1771$.

		\item The $\mathcal{S}_{sin}^{*}$ radius
		$R_{\mathcal{S}_{sin}^{*}}= \sin 1/\left(2+ \sqrt{4+ \sin 1 (2+ \sin 1)}\right) \approx 0.1858$.
		
		\item The $\mathcal{S}_{\leftmoon}^{*}$ radius
		$R_{\mathcal{S}_{\leftmoon}^{*}}= \sqrt{2}-\sqrt{(3-\sqrt{2})} \approx 0.1549$.
		
		\item The $\mathcal{S}_{R}^{*}$ radius
		$R_{\mathcal{S}_{R}^{*}}= (3-2\sqrt{2})/\left(2+\sqrt{15-8\sqrt{2}} \right)\approx 0.0438$.
		
		\item The $\mathcal{S}_{RL}^{*}$ radius
		$R_{\mathcal{S}_{RL}^{*}}$ is the root $(\approx 0.0694)$ in $[0,1]$ of the equation
		\begin{align*} 16r^{2}+ (r^{2}-1)^2-(1-\sqrt{2}+(1+\sqrt{2})r^{2})^{2}+ &\\
		(r^2-1)
		\sqrt{2\sqrt{2}-2-2(1+\sqrt{2})r^{4}}= 0.
		\end{align*}
		
		\item The $\mathcal{S}_\gamma^{*}$ radius  $R_{S_\gamma^{*}} \geq (\sin(\pi\gamma/2))/\left(2+\sqrt{4-\sin^{2}(\pi\gamma/2)}\right), \quad 0 < \gamma \leq 1$.
		
	\end{enumerate}
\end{theorem}
\begin{proof}
	
	Let the function $ f \in \mathcal{F}_{4}$. Then
	\begin{equation} \label{g4}
	\operatorname{Re}\left (\frac{(1+z)^{2}}{z}f(z)\right )>0  \quad (z \in \mathbb{D}).
	\end{equation}
	Define the function $h:\mathbb{D} \longrightarrow \mathbb{C} $ by
	\begin{equation}\label{g6}
	h(z)=\frac{(1+z)^{2}}{z}f(z).
	\end{equation}
	By (\ref{g4}) and (\ref{g6}) we have $h \in \mathcal{P}$
	and $f(z)= z h(z)/(1+z)^{2} $.
	Therefore,  by calculation it can be shown that
	\begin{equation} \label{main4}
	\frac{zf'(z)}{f(z)}=\frac{z h^{'}(z)}{h(z)}+\frac{1-z}{1+z}.
	\end{equation}
	The bilinear transformation $(1-z)/(1+z)$ maps the disk $\left|z\right| \leq r$ onto the disk
	\begin{equation}\label{disk4}
	\left | \frac{1-z}{1+z}- \frac{1+r^{2}}{1-r^{2}}  \right | \leq \frac{2r}{1-r^{2}}.
	\end{equation}
	Using (\ref{disk4}) and (\ref{shah}),   it follows from (\ref{main4}) that $f$ maps the disk $\left|z\right| \leq r$ onto the disk
	\begin{equation} \label{star4}
	\left | \frac{zf'(z)}{f(z)} - \frac{1+r^{2}}{1-r^{2}}  \right |  \leq \frac{4r}{1-r^{2}}.
	\end{equation}
	The classes we discuss here are all subclasses of starlike functions. By \eqref{star4}, we have
	\begin{equation} \label{starlike4}
		\operatorname{Re} \frac{zf'(z)}{f(z)} \geq \frac{1-4r+r^{2}}{1-r^{2}} \geq 0, \quad (r \leq 2-\sqrt{3}).
	\end{equation}
	Hence all the radii that we estimate will be less than $2-\sqrt{3} \approx 0.26794$. Note that  for $0 < r \leq (2-\sqrt{3})$,  the centre of disk in (\ref{star4}) lies in the interval $[1,  (4-2\sqrt{3})/(2\sqrt{3}-3)] \approx [1, 1.1547]$.
	
	\begin{enumerate}[label=(\roman*)]
		
		\item
		The number $r=R_{\mathcal{S}^{*}(\alpha)} $ is the root of
		$  (1+\alpha) r^{2}-4r+1-\alpha=0$ in $[0,1]$ and hence, for $0<r\leq R_{\mathcal{S}^{*}(\alpha)}$, it follows
		by (\ref{starlike4}) that
		\[ 	\operatorname{Re} \frac{zf'(z)}{f(z)} \geq   \frac{1-4r+r^2}{1-r^{2}}\geq \alpha.\]
		For the function  $f_{4} \in \mathcal{F}_{4}$  given by \eqref{extremal4}  we have
		\[\frac{zf_{4}^{'}(z)}{f_{4}(z)}=\frac{1-4z+z^2}{1-z^{2}}=\frac{1-4r+r^2}{1-r^{2}}=\alpha,\quad (z=r= R_{\mathcal{S}^{*}(\alpha)})\]
		and this shows that the radius is sharp.

		\item	
		It   follows from (\ref{star4}) that
		\begin{equation} \label{estimate4}
		\left | \frac{zf'(z)}{f(z)} - 1  \right | \leq \left|\frac{zf'(z)}{f(z)}-\frac{1+r^2}{1-r^2}\right|+\frac{2r^2}{1-r^2}  \leq \frac{4r+ 2r^{2}}{1-r^{2}}.
		\end{equation} The number $r=R_{\mathcal{S}_{L}^{*}}$ is the root in $[0,1]$ of $(4r+ 2r^{2})\leq (\sqrt{2}-1)(1-r^{2})$ and for $0 < r \leq R_{\mathcal{S}_{L}^{*}}$,  we have
		\begin{equation} \label{lemniradius4}
		\frac{4r+ 2r^{2}}{1-r^{2}} \leq \sqrt{2}-1.
		\end{equation}
		Therefore,  by (\ref{estimate4}) and (\ref{lemniradius4}), for $0 < r \leq R_{\mathcal{S}_{L}^{*}}$, we have
		\begin{equation}\label{connect4}
		\left | \frac{zf'(z)}{f(z)} - 1  \right |  \leq \frac{4r+ 2r^{2}}{1-r^{2}} \leq \sqrt{2}-1.
		\end{equation}
		For $0 < r \leq R_{\mathcal{S}_{L}^{*}}$,  using triangle inequality and \eqref{connect4}, we have
		\begin{equation}\label{tconnect4}
			\left | \frac{zf'(z)}{f(z)} + 1  \right | \leq \sqrt{2}+1
		\end{equation}
		and hence by \eqref{connect4} and \eqref{tconnect4}
		\[ \left|\left(\frac{zf'(z)}{f(z)}\right)^{2} - 1\right|\leq \left|\frac{zf'(z)}{f(z)} + 1\right|\left|\frac{zf'(z)}{f(z)} - 1\right|\leq (\sqrt{2}+1)(\sqrt{2}-1)=1.\]
		The number  $\rho=  R_{\mathcal{S}_{L}^{*}}$ satisfies $(1+4\rho-\rho^{2})/(1-\rho^{2})= \sqrt{2}$. Using this, we see that at $z:= -\rho= -R_{\mathcal{S}_{L}^{*}}$, the  function $f_{4}$ defined in (\ref{extremal4}) satisfies
		\[ \left| \left(\frac{zf_{4}'(z)}{f_{4}(z)}\right)^{2} - 1    \right|=   \left| \left(\frac{1-4z+z^{2}}{1-z^{2}}\right )^{2}-1\right|= \left| \left(\frac{1+4\rho+\rho^{2}}{1-\rho^{2}}\right)^{2} - 1    \right|=1.\]
		This shows that the radius is sharp.
	
			\item
			If $0 < r \leq R_{\mathcal{S}_{p}}$,  then $a= 1/(1-r^{2}) \leq 3/2$ and
			\[\frac{4r}{1-r^{2}} \leq \frac{1+r^{2}}{1-r^{2}}-\frac{1}{2}.\] Thus,  by \eqref{parabolic},  we see that the disk in (\ref{star4}) lies inside the parabolic region $\Omega_{PAR}$. Sharpness follows for the function $f_{4}$ defined in (\ref{extremal4}) (See Figure \ref{fig:pb4}.(\subref{fig4:sub1})).
			At $z:=\rho= R_{\mathcal{S}_{p}}$, we have
			\[\operatorname{Re}\frac{zf'(z)}{f(z)}=\frac{1-4\rho+\rho^2}{1-\rho^2}=\frac{4\rho-2\rho^2}{1-\rho^2}=\left|\frac{2\rho^2-4\rho}{1-\rho^2}\right|=\left|\frac{zf'(z)}{f(z)}-1\right|.\]
			\begin{figure}
				\centering
				\begin{subfigure}{.5\textwidth}
					\centering
					\includegraphics[height=3cm,keepaspectratio]{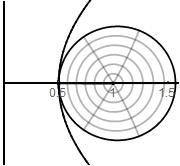}
					\caption{Sharpness of class $\mathcal{S}_{p}$}
					\label{fig4:sub1}
				\end{subfigure}%
				\begin{subfigure}{.5\textwidth}
					\centering
					\includegraphics[height=3cm,keepaspectratio]{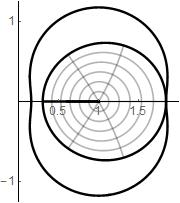}
					\caption{Sharpness of class $\mathcal{S}_{sin}^{*}$}
					\label{fig4:sub2}
				\end{subfigure}
				\caption{Sharpness of parabolic starlike functions and starlike functions associated with sine function}
				\label{fig:pb4}
			\end{figure}	 	
		
		\item
		For $0 < r \leq R_{\mathcal{S}_{e}^{*}}$, we have
		$1/e \leq  a= 1/(1-r^{2}) \leq (\ e+\ e^{-1})/2$ and
		\[\frac{4r}{1-r^{2}} \leq \frac{1+r^2}{1-r^{2}}-\frac{1}{e}.\]
		By  \eqref{exponential}, the disk in \eqref{star4}  lies inside
		$ \Omega_{e} $ for $0 < r \leq R_{\mathcal{S}_{e}^{*}}$ proving   that  the $\mathcal{S}_{e}^{*}$ radius for the class $\mathcal{F}_{4}$ is $R_{\mathcal{S}_{e}^{*}}$.
		The sharpness follows for the function $f_{4}$ defined in (\ref{extremal4}).
		Indeed at $z:=\rho= R_{\mathcal{S}_{e}^{*}}$, we have  \[ \left| \log \frac{zf_{4}'(z)}{f_{4}(z)} \right|= \left| \log \frac{1-4 \rho+\rho^2}{1-\rho^{2}}\right|=1.\]

		\item
		If $0 < r \leq R_{\mathcal{S}_{c}^{*}}$, then
		 	\[\frac{4r}{1-r^{2}} \leq  \frac{1+r^{2}}{1-r^{2}}-\frac{1}{3} .\] By (\ref{cardiod}), we see that the disk in (\ref{star4}) lies inside $\Omega_{c}$,  if $0 < r \leq R_{\mathcal{S}_{c}^{*}}$. The  result is sharp for the function $f_4$ defined in (\ref{extremal4}).  At $z:=\rho= R_{\mathcal{S}_{c}^{*}}$,
		 	\[\left|\frac{zf'(z)}{f(z)}\right|=\left|\frac{1-4\rho+\rho^2}{1-\rho^2}\right|=\frac{1}{3}=\Omega_c(-1)\in\partial \Omega_c(\mathbb{D}).\]

		\item
		For $0 < r \leq R_{\mathcal{S}_{sin}^{*}}$
		 	\[\frac{4r}{1-r^{2}} \leq \sin1-\frac{2r^{2}}{1-r^{2}} .\] It is evident from (\ref{sine}) that the disk in (\ref{star4}) lies inside $\Omega_{s}$ provided $0 < r \leq R_{\mathcal{S}_{sin}^{*}}$. For the function $f_{4}$ defined in (\ref{extremal4}) (See Figure \ref{fig:pb4}.(\subref{fig4:sub2})), at $z:=-\rho= -R_{\mathcal{S}_{sin}^{*}}$,
		 	\[\left|\frac{zf'(z)}{f(z)}\right|=\left|\frac{1+4\rho+\rho^2}{1-\rho^2}\right|=1+\sin 1=q_0(1)\in\partial \Omega_s(\mathbb{D}).\]

		\item
	If $0 < r \leq R_{\mathcal{S}_{\leftmoon}^{*}}$, then
	\[\frac{4r-1-r^{2}}{1-r^{2}} \leq 1-\sqrt{2}.\]
Thus, by \eqref{lune}, the disk in \eqref{star4} lies inside $\{w \in \mathbb{C}: \left | w^{2}-1 \right |< 2\left|w\right| \}$ and hence $f \in \mathcal{S}_{\leftmoon}^{*}$.
	The sharpness follows from the function defined in (\ref{extremal4}).
	At $z:=\rho= R_{\mathcal{S}_{\leftmoon}^{*}}$, we have
	\[ \left| \left(\frac{zf_{4}'(z)}{f_{4}(z)}\right)^{2} - 1    \right|= \left| \left(\frac{\rho^{2}-4\rho+1}{1-\rho^{2}}\right)^{2} - 1    \right|=2\left| \left(\frac{\rho^{2}-4\rho+1}{1-\rho^{2}}\right )\right|=2\left|\left(\frac{zf_{4}'(z)}{f_{4}(z)}\right)\right|.\]

		\item
		If $0 < r \leq R_{\mathcal{S}_{R}^{*}}$, $2(\sqrt{2}-1) < a= 1/(1-r^{2}) \leq \sqrt{2}$ and
		\[\frac{4r-1-r^{2}}{1-r^{2}} \leq 2-2\sqrt{2},  \quad 0 < r \leq R_{\mathcal{S}_{R}^{*}}.\] Then, by \eqref{rational}, the disk in (\ref{star4}) lies inside $\psi(\mathbb{D})$. The result is sharp for the function defined in (\ref{extremal4}) (See Figure \ref{fig:pb4a}.(\subref{fig4:sub3})).  At $z:=\rho= R_{\mathcal{S}_{R}^{*}}$,
		\[\left|\frac{zf'(z)}{f(z)}\right|=\left|\frac{1-4\rho+\rho^2}{1-\rho^2}\right|=2(\sqrt{2}-1)=\psi(1)\in\partial \psi(\mathbb{D}).\]

		\item
		If $0 < r \leq R_{\mathcal{S}_{RL}^{*}}$, $ \sqrt{2}/3 \leq a= 1/(1-r^{2}) < \sqrt{2}$, and
			\begin{align*} 16r^{2} -(1-r^{2})\sqrt{(1-r^{2})^{2}-((\sqrt{2}-\sqrt{2}r^{2})-(1+r^{2})^{2}}&\\
		 +         (1-r^{2})^{2}-((\sqrt{2}-\sqrt{2}r^{2})-(1+r^{2})^{2}&\leq 0.
		\end{align*}
		Then, by \eqref{reverse}, the disk in \eqref{star4} lies inside the region $\{w: | (w-\sqrt{2})^{2}-1 |< 1\}$. The result is sharp for the function defined in (\ref{extremal4}) (See Figure \ref{fig:pb4a}.(\subref{fig4:sub4})).
			\begin{figure}
			\centering
			\begin{subfigure}{.5\textwidth}
				\centering
				\includegraphics[height=3cm,keepaspectratio]{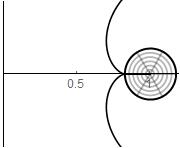}
				\caption{Sharpness of class $\mathcal{S}_{R}^{*}$}
				\label{fig4:sub3}
			\end{subfigure}%
			\begin{subfigure}{.5\textwidth}
				\centering
				\includegraphics[height=3cm,keepaspectratio]{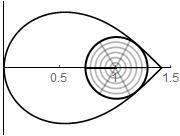}
				\caption{Sharpness of class $\mathcal{S}_{RL}^{*}$}
				\label{fig4:sub4}
			\end{subfigure}
			\caption{Sharpness of starlike functions associated with a rational function and reverse lemniscate.}
			\label{fig:pb4a}
		\end{figure}	 	
		\item
		If $0 < r \leq R_{\mathcal{S}_\gamma^{*}}$, $0 < \gamma \leq 1$, then
$ r^{2}\sin(\pi\gamma/2)-4r+\sin(\pi\gamma)/2)\leq 0$.
		It is evident from \eqref{strongstar} that the disk (\ref{star4}) is contained in the sector $\left| \arg w \right| \leq (\pi\gamma)/2, $ if $0 < r \leq R_{\mathcal{S}_\gamma^{*}}$. \qedhere

	\end{enumerate}
\end{proof}


\end{document}